\documentclass[12pt]{article}

\usepackage[T1]{fontenc}
\usepackage{avant}
\usepackage[cp1250]{inputenc}
\usepackage{amsmath,amssymb,amsthm}
\usepackage{geometry}
\usepackage{caption}
\usepackage{float}
\usepackage{graphicx}

\usepackage{makeidx}
\usepackage[matrix,arrow,curve]{xy}
\usepackage{boxedminipage}
\usepackage{epic}
\usepackage{longtable}
\usepackage{ifthen}
\usepackage{tabularx}
\usepackage{parskip}
\usepackage{verbatim}
\usepackage{multirow}

\usepackage[ruled,linesnumbered,norelsize]{algorithm2e}
\usepackage{authblk}

\newtheorem{theorem}{Theorem}
\newtheorem{proposition}{Proposition}
\newtheorem{lemma}{Lemma}
\newtheorem{definition}{Definition}

\newtheorem{corollary}{Corollary}
\newtheorem{remark}{Remark}

\title{Homogeneous spaces  of real simple Lie groups with proper actions of non virtually abelian discrete subgroups: a calculational approach}
\author{Maciej Boche\'nski, Piotr Jastrz\c ebski and  Aleksy Tralle}

\begin{document}

\maketitle

\abstract{Let $G$ be a simple non-compact linear connected Lie group  and $H\subset G$ be a closed  non-compact semisimple subgroup.  We are intersted in finding classes of homogeneous spaces $G/H$ admitting proper actions of discrete non-virtually abelian subgroups $\Gamma\subset G$.  We develop an algorithm for finding such homogeneous spaces. As a testing example we obtain a list of all non-compact homogeneous spaces $G/H$ admitting proper action of a discrete and non virtually abelian subgroup $\Gamma \subset G$ in the case when $G$ has rank at most 8, and $H$ is a maximal proper semisimple subgroup.}


\section{Introduction}\label{sec:intro}

A group is called non-virtually abelian if it has no finite index abelian subgroups. Let $G$ be a real simple linear non-compact  Lie group and let $H\subset G$ be a proper closed non-compact  semisimple subgroup. In this paper we are interested in a problem of finding homogeneous spaces  $G/H$, which admit a proper action of non virtually abelian discrete subgroups $\Gamma\subset G$. We create a procedure of generating some homogeneous spaces with the aformentioned property. As a testing example,  we give a complete list of non-compact homogeneous spaces $G/H$ which admit a proper action of a discrete subgroup $\Gamma\subset G$ which is non-virtually abelian when $H$ is a maximal proper semisimple subgroup and the rank of $G$ is $\leq 8$.

Let us recall the definition of a proper group action.
Let $L$ be a locally compact topological group acting continuously on a locally Hausdorff topological space $X$. This action is {\it proper} if for every compact subset $S \subset X$ the set
$$L_{S}:=\{  g\in L | g\cdot S \cap S \neq \emptyset \}$$
is compact.


If there exists a discrete $\Gamma\subset G$ acting properly on $G/H$, we say that $G/H$ admits a Clifford-Klein form. If $\Gamma\backslash G/H$ is compact, we say that $G/H$ admits a compact Clifford-Klein form.  The problem of finding an appropriate discrete subgroup is straightforward when  $H$ is compact. In this case any torsion-free discrete subgroup of $G$ acts properly on $G/H$. On the other hand if $H$ is non-compact then it may happen that no infinite discrete subgroup of G acts properly  on $G/H.$ In more detail, if $\text{rank}_{\mathbb{R}}G=\text{rank}_{\mathbb{R}}H$ then only finite groups can act properly on $G/H$  (this  is  the Calabi-Markus phenomenon \cite{ko}, \cite{K89}).

\noindent
Moreover there are spaces which only admit a proper action of virtually abelian discrete subgroups (see \cite{ben}), for instance $SL(2n+1,\mathbb{R})/SL(2n,\mathbb{R})$.  Benoist \cite{ben} has found a criterion of the existence of a proper actions of a non-virtually abelian discrete subgroup of $G$ on $G/H$ expressed in terms of the Lie algebra data. This criterion yielded explicit examples of $G/H$ which admit and which do not admit such actions. Using the results of Benoist and T. Kobayashi, and a classification of nilpotent orbits, Okuda \cite{ok} classified irreducible symmetric spaces admitting proper action of non-virtually abelian discrete subgroups.

\noindent
Motivated by these results, the first and the third named author of this article proposed in \cite{bt} a definition of an a-hyperbolic rank of a reductive real Lie group. Using the a-hyperbolic ranks one can formulate conditions (similar to the Calabi-Markus phenomenon) under which the space $G/H$ admits, or does not admit proper actions of non-virtually abelian discrete subgroups (\cite{bt}, Theorem \ref{th8}). In \cite{btjo} other examples of homogeneous spaces $G/H$ admitting proper actions of non-virtually abelian discrete subgroups were given.
 
As explained in \cite{kod}, Section 3.8, in general the smaller the non-compact part of $H$, the more possibilities for discrete subgroups $\Gamma\subset G$ proper actions on $G/H$ exist. Therefore in this paper we investigate the extreme case, that is, when $H$ is a maximal proper subgroup of $G.$ One expects that in this situation the existence of an appropriate $\Gamma$ is the most difficult to obtain.

Since the problem of finding discrete subgroups of $G$ acting properly on $G/H$ is difficult, it is tempting to take another route in a more algorithmic fashion. T. Kobayashi found a criterion of properness of the action of reductive Lie subgroup $L\subset G$ on a homogeneous space $G/H$ expressed in purely Lie-theoretic terms \cite{K89}. We present this result as Theorem \ref{critkob}. Note that if $\Gamma$ is a  lattice in $L$, then it also acts properly on $G/H$. However, the problem of classifying  triples $(G,H,L)$ satisfying the criterion \cite{K89} is far from being solved, although there are numerous examples of both types, satisfying and not satisfying it.  Since a general description of such triples $(G,H,L)$ is not well understood, it is natural to take up  a computer aided investigation of various aspects of  proper actions of Lie groups on homogeneous spaces. In this direction we developed several numerical procedures. For example, in \cite{bjt} we proved that there are no compact Clifford-Klein forms of exceptional Lie groups. In \cite{bjstw} we proposed an algorithm for checking an obstruction to the existence of compact Clifford-Klein forms.   Looking at the problem from this angle we keep in mind that there are algorithms for various calculations in semisimple Lie algebras designed by de Graaf et. al., which yield methods of construction of semisimple subalgebras in simple complex Lie algebras \cite{dg1}, in real Lie algebras \cite{fd}, calculations with the real  Weyl groups \cite{ddg}, and the classification of real semisimple subalgebras in simple real Lie algebras of rank $\leq 8$ \cite{dg}. It should be mentioned that these algorithms were implemented (see packages \cite{gap} and \cite{sla}). 
In this article we develop an algorithm which checks the Kobayashi's criterion for the case when the real rank of $L$ is $1$ and $\mathfrak{g}, \mathfrak{l}$ are split Lie algebras (Theorem \ref{thm:procedure-split}). We apply our algorithms to the database \cite{dg} and obtain the complete description of homogeneous spaces $G/H$ of simple real linear Lie group $G$ of rank $\leq 8$ determined by maximal semisimple subgroups $H$ in $G$ (Theorem \ref{thm:maximal8}) such that $G/H$ admits proper action of a discrete non-virtually abelian subgroup of $G$. We implement our algorithms using packages \cite{gap},\cite{corelg} and \cite{sla}.

\noindent {\bf Acknowledgment}. We thank Willem de Graaf for discussions  and answering our questions. The first named author acknowledges the support of the National Science Center, Poland (grant NCN no. 2018/31/D/ST1/00083). The third named author was supported by National Science Center, Poland (grant  2018/31/B/ST1/00053).

\section{Preliminaries}
Our basic references for  the Lie theory are \cite{ov2} and \cite{OV}. We consider real and complex Lie algebras and we write $\mathfrak{g}^c$ for a complex Lie algebra and $\mathfrak{g}$ for the real form of $\mathfrak{g}^c$. We say that $\mathfrak{h}$ is a maximal subalgebra in $\mathfrak{g}$ if $\mathfrak{h}^c$ is maximal and proper in $\mathfrak{g}^c$ (with respect to inclusion). In general, it may happen that there is a subalgebra $\mathfrak{h}\subset\mathfrak{g}$ which is maximal in $\mathfrak{g}$ but $\mathfrak{h}^c$ is not maximal in $\mathfrak{g}^c$. Therefore, we restrict ourselves to a smaller class of inclusions of Lie algebras. Now we just fix some notation used throughout the article and recall the notion of the Satake diagram.  Let $\mathfrak{g}$ be a semisimple real Lie algebra. There is a Cartan decomposition $\mathfrak{g}=\mathfrak{k}\oplus\mathfrak{p}$ and the corresponding Cartan involution $\theta$. Fix a split Cartan subalgebra $\mathfrak{j}=\mathfrak{t}\oplus\mathfrak{a}$, $\mathfrak{a}\subset\mathfrak{p}$ being a maximal abelian subalgebra in $\mathfrak{p}$. Then $\mathfrak{j}^c$ is a Cartan subalgebra in $\mathfrak{g}^c$. Take a root system $\Delta=\Delta(\mathfrak{g}^c,\mathfrak{j}^c)$ with respect to $\mathfrak{j}^c$ and define
$$\Delta_0=\{\alpha\in\Delta\,|\,\alpha|_{\mathfrak{a}}=0\}.$$
Put $\Delta_1=\Delta\setminus\Delta_0$ and choose a set of simple roots $\Pi$ in $\Delta$. Put $\Pi_0=\Pi\cap\Delta_0$ and $\Pi_1=\Delta_1\cap\Pi$. Define an involution $\sigma^*:\mathfrak{j}^*\rightarrow\mathfrak{j}^*$ by the formula
$(\sigma^*\mu)(x)=\overline{\mu(\sigma(x))}$. One can show that $\sigma^*$ determines an involution $\tilde\sigma$ on $\Pi_1$. Then one defines the Satake diagram as follows: one takes the Dynkin diagram for $\mathfrak{g}^c$ and paints vertices from $\Pi_0$ in black and from $\Pi_1$ in white. Then one joins by arrow white vertices transformed into each other by $\tilde\sigma$. Recall that Satake diagrams classify semisimple real Lie algebras up to isomorphism. The list of all Satake diagrams of real simple Lie algebras can be found in \cite{OV}.
 \begin{definition} {\rm A weighted Dynkin diagram of a vector $x\in\mathfrak{j}^c$ is a map $\psi_x:\Pi\rightarrow \mathbb{R}$ defined by $\psi_x(\alpha)=\alpha(x)$ for any $\alpha\in\Pi$.}
\end{definition}
\noindent We say that $\mathfrak{g}$ is split, if its Satake diagram coincides with the Dynkin diagram (thus, there are no black nodes). For example, $\mathfrak{sl}(n,\mathbb{R})$ is split.
\begin{definition}{\rm We say that a weighted Dynkin diagram $\psi_x$ matches the Satake diagram of $\mathfrak{g}$ if all black nodes have weights equal to $0$ in $\psi_x$ and every two nodes joined by an arrow have the same weights.}
\end{definition}

\subsection{Proper actions and the a-hyperbolic rank}
We begin with two tools which often detect the existence or non-existence of proper actions  of reductive subgroups on a homogeneous space $G/H$ . The first one is Kobayashi's criterion of properness \cite{K89}, the second is an a-hyperbolic rank \cite{bt}.  Passing to the conjugates if necessary, we assume that there exists a Cartan involution $\theta$ of $\mathfrak{g}$ such that $\theta (\mathfrak{h})=\mathfrak{h}$ and $\theta (\mathfrak{l})=\mathfrak{l} .$ We get the following Cartan decompositions
$$\mathfrak{g} = \mathfrak{k} + \mathfrak{p}, \ \ \mathfrak{h} = \mathfrak{k}_{h} + \mathfrak{p}_{h},  \ \ \mathfrak{l} = \mathfrak{k}_{l} + \mathfrak{p}_{l}.$$
Let $\mathfrak{a} , \ \mathfrak{a}_{h} , \ \mathfrak{a}_{l}$ be maximal abelian subspaces in $\mathfrak{p}, \ \mathfrak{p}_{h}, \ \mathfrak{p}_{l} ,$ respectively. We may assume that $\mathfrak{a}_{h}, \mathfrak{a}_{l}\subset \mathfrak{a}.$ Also denote by $W(\mathfrak{a})$ the little Weyl group of $\mathfrak{g} .$
\begin{theorem}[\cite{K89}, Theorem 4.1]
The action of $L$ on a homogeneous space $G/H$  is proper if and only if
$$W(\mathfrak{a})\mathfrak{a}_{l}\cap \mathfrak{a}_{h} = \{  0 \} .$$
\label{critkob}
\end{theorem}
\begin{corollary}[The Calabi-Markus phenomenon, \cite{K89}, Corollary 4.4]
If 
$$\operatorname{rank}_{\mathbb{R}}(\mathfrak{g})=\operatorname{rank}_{\mathbb{R}}(\mathfrak{h}),$$ 
 then only finite subgroups of $G$ act properly on $G/H.$
\end{corollary}
\noindent There is one more tool of checking if $G/H$ admits or does not admit proper actions of non-virtually abelian discrete subgroups.  It does not depend on the embedding $\mathfrak{h}\hookrightarrow\mathfrak{g}$.  Let $\Sigma^+$ be the set of positive restricted roots of $\mathfrak{g}$ with respect to $\mathfrak{a}$. Let $w_0\in W(\mathfrak{g}^c)$ denote the longest element in the Weyl group of $\mathfrak{g}^c$.   It determines an involutive automorphism $\iota$ of the Cartan subalgebra of $\mathfrak{g}$ given by the formula $X\rightarrow -(w_0X)$. Let $\mathfrak{a}^+$ denote the subset 
$$\mathfrak{a}^+=\{H\in\mathfrak{a}\,|\,\alpha(H)\geq 0,\,\forall \alpha\in\Sigma^+\}.$$
Consider the convex cone
$$\mathfrak{b}^+=\{H\in\mathfrak{a}^+\,\,|\,\iota(H)=H\}.$$
\begin{definition} {\rm An a-hyperbolic rank of $\mathfrak{g}$ is defined as} 
$$\operatorname{rank}_{a-hyp}\mathfrak{g}=\dim\mathfrak{b}^+.$$
\end{definition}

\begin{theorem}[ \cite{bt}, Theorem 8] Under the assumptions of Theorem \ref{critkob}
	\begin{enumerate}
		\item If $\operatorname{rank}_{\textrm{a-hyp}} \mathfrak{g} = \operatorname{rank}_{\textrm{a-hyp}} \mathfrak{h}$ then $G/H$ does not admit proper actions of non virtually abelian discrete subgroups.
		\item If $\operatorname{rank}_{\textrm{a-hyp}} \mathfrak{g} > \operatorname{rank}_{\mathbb{R}} \mathfrak{h}$ then $G/H$ admits a proper action of a subgroup $L\subset G$ locally isomorphic to $\mathfrak{sl}(2,\mathbb{R}) .$
	\end{enumerate}
\label{th8}
\end{theorem}

\noindent The a-hyperbolic rank of simple real Lie algebras can be calculated using the data in Table \ref{tab1}. Table \ref{tab1} can also be used to calculate the a-hyperbolic rank of a reductive real Lie algebra since the following holds:
\begin{itemize}
	\item the a-hyperbolic rank of a semisimple Lie algebra equals the sum of a-hyperbolic ranks of all its simple parts.
	\item the a-hyperbolic rank of a reductive Lie algebra equals the a-hyperbolic rank of its derived subalgebra.
\end{itemize}

\begin{center}
 \begin{table}[h]
 \centering
 {\footnotesize
 \begin{tabular}{| c | c | c |}
   \hline
   \multicolumn{3}{|c|}{ \textbf{\textit{a-hyperbolic ranks of simple Lie algebras}}} \\
   \hline                        
   $\mathfrak{g}$ & $\text{rank}_{a-hyp} (\mathfrak{g})$ & $\text{rank}_{\mathbb{R}} (\mathfrak{g})$ \\
   \hline
   $\mathfrak{sl}(2k,\mathbb{R})$, $\mathfrak{sl}(2k,\mathbb{C})$  & k &  2k-1 \\
   {\scriptsize $k\geq 2$} & & \\
   \hline
   $\mathfrak{sl}(2k+1,\mathbb{R})$, $\mathfrak{sl}(2k+1,\mathbb{C})$  & k & 2k \\
   {\scriptsize $k\geq 1$} & & \\
   \hline
   $\mathfrak{su}^{\ast}(4k)$  & k & 2k-1 \\
   {\scriptsize $k\geq 2$} & & \\
   \hline
   $\mathfrak{su}^{\ast}(4k+2)$  & k & 2k \\
   {\scriptsize $k\geq 1$} & & \\
   \hline
   $\mathfrak{so}(2k+1,2k+1)$  & 2k & 2k+1 \\
   {\scriptsize $k\geq 2$} & &  \\
   \hline
	 $\mathfrak{e}_{6}^{\text{I}}$ & 4 & 6 \\
	 \hline
   $\mathfrak{e}_{6}^{\text{IV}}$  & 1 & 2 \\
   \hline  
	$\mathfrak{so}(4k+2,\mathbb{C})$  & 2k & 2k+1 \\
	{\scriptsize $k\geq 2$} & & \\
   \hline 
	$\mathfrak{e}_{6}^{\mathbb{C}}$  & 4 & 6 \\
   \hline 
 \end{tabular}
 }
\captionsetup{justification=centering}
 \caption{
 The table contains all simple real Lie algebras $\mathfrak{g},$ for which 
$\operatorname{rank}_{\mathbb{R}}(\mathfrak{g}) \neq \operatorname{rank}_{a-hyp}(\mathfrak{g})$}
 \label{tab1}
 \end{table}
\end{center}

A triple $(h,e,f)$ of vectors in $\mathfrak{g}$ is called an  $\mathfrak{sl}(2,\mathbb{R})$-triple  if

$$[h,e]=2e, \ \ [h,f]=-2f, \ \mathrm{and} \ [e,f]=h.$$
If $G$ is non-compact, there exists a homomorphism of Lie groups
$\nu : SL(2,\mathbb{R}) \rightarrow G$,
determined by 
$$d\nu  \left( \begin{array}{cc} 1 & 0  \\ 0 & -1  \end{array} \right) = h, \ \ d\nu  
\left( \begin{array}{cc} 0 & 1  \\ 0 & 0  \end{array} \right) = e, \ \ d\nu  \left( 
\begin{array}{cc} 0 & 0  \\ 1 & 0  \end{array} \right) = f.$$

The same definition applied to  $\mathfrak{sl}(2,\mathbb{C})$ yields an $\mathfrak{sl}(2,\mathbb{C})$-triple in a complex semisimple Lie algebra $\mathfrak{g}^c$.

\subsection{Presentation of the little Weyl group} 
It is convenient for us to follow \cite{FH}, although there are other sources of the material presented here. In this section $G$ denotes a reductive algebraic group over an arbitrary field $k$ of characteristic not $2$, $\theta$ denotes an involution in $\operatorname{Aut}(G)$, and 
$$K=G^{\theta}=\{g\in G\,|\,\theta(g)=g\}.$$

Let $\mathfrak{g}$ denote the Lie algebra of $G$. The involution $\theta\in\operatorname{Aut}(G)$ induces an involution in $\operatorname{Aut}(\mathfrak{g})$, which is denoted by the same letter. We get a decomposition
$$\mathfrak{g}=\mathfrak{g}^{\theta}\oplus\mathfrak{p}=\mathfrak{k}\oplus\mathfrak{p},$$
$$\mathfrak{k}=\mathfrak{g}^{\theta}=\{\theta(A)=A,\,|\,A\in\mathfrak{g}\},\,\mathfrak{p}=\{A\in\mathfrak{g}\,|\,\theta(A)=-A\}.$$
For a toral subalgebra $\mathfrak{t}\subset\mathfrak{g}$ one defines the set of roots $\Phi(\mathfrak{t})$ in a standard way, as well as the root decomposition
$$\mathfrak{g}=\mathfrak{g}_0\oplus\sum_{\alpha\in\Phi(\mathfrak{t})}\mathfrak{g}_{\alpha}, \ \  \mathfrak{g}_0=Z_{\mathfrak{g}}(\mathfrak{t}).$$
If $\mathfrak{t}$ is a {\it maximal} toral subalgebra, then  $\Phi(\mathfrak{t})$ is a (reduced) root system. In general, if $\mathfrak{t}$ is not maximal, this is not the case. However, if $\mathfrak{a}\subset\mathfrak{p}$ is a maximal toral subalgebra in $\mathfrak{p}$, $\Phi(\mathfrak{a})$ is a root system as well. In this section we always assume the following:
\begin{itemize}
\item $\mathfrak{t}$ is a maximal toral subalgebra of $\mathfrak{g}$, such that $\mathfrak{a}\subset\mathfrak{t}$, where $\mathfrak{a}\subset\mathfrak{p}$ is a maximal toral subalgebra in $\mathfrak{p}$;
\item $\Phi(\mathfrak{t})$, $\Phi(\mathfrak{a})$ denote the root systems determined by $\mathfrak{t}$ and $\mathfrak{a}$, respectively.
\end{itemize}
\begin{proposition}[\cite{FH}, Lemma 1]\label{prop:theta-stab} In the above notation:
\begin{enumerate}
\item $\mathfrak{t}$ is $\theta$-stable,
\item $\mathfrak{t}=\mathfrak{t}_{+}\oplus\mathfrak{t}_{-}$, where $\mathfrak{t}_{\pm}=\{x\in\mathfrak{t}\,|\,\theta(x)=\pm x\}$,
\item $\mathfrak{t}_{-}=\mathfrak{a}$, $\mathfrak{t}_{+}\subset\mathfrak{k}$.
\end{enumerate}
\end{proposition}
Thus, $\Phi(\mathfrak{t})$ and $\Phi(\mathfrak{a})$ are root systems. The relation between them is described  in \cite{FH} as follows. Let 
$$R(\mathfrak{t})=\mathbb{Z}_{span}(\Phi(\mathfrak{t})).$$
Note that $\theta$ acts on $\mathfrak{t}$, and, therefore, on $R(\mathfrak{t})$.
Introduce the following notation:
\begin{itemize}
\item  $\Delta(\mathfrak{t})$ is a basis of $\Phi(\mathfrak{t})$ and $\Delta(\mathfrak{a})$ s a basis of $\Phi(\mathfrak{a})$,
\item $\Phi_0=\{\alpha\in\Phi(\mathfrak{t})\,|\,\alpha|_{\mathfrak{a}}=0\}$,
\item $X_0(\theta)=\{\chi\in R(\mathfrak{t})\,|\,\theta(\chi)=\chi\}$, and $\Phi_0(\theta)=\Phi(\mathfrak{t})\cap X_0(\theta)$
\end{itemize}
\begin{proposition}[\cite{FH}, Lemma 5]\label{prop:phi0} In the above notation,
\begin{enumerate}
\item $\Phi_0$ and $X_0$ are $\theta$-stable,
\item $\Phi_0(\theta)$ is a closed subsystem of $\Phi(\mathfrak{t})$,
\item $\Phi_0=\Phi_0(\theta)$.
\end{enumerate}
\end{proposition}
For a subset $S\subset\Phi(\mathfrak{t})$ denote by $W(S)$ the subgroup of the Weyl group $W(\Phi(\mathfrak{t}))$ generated by reflections $s_{\alpha},\alpha\in S$. Introduce the following subgroups of $W(\Phi(\mathfrak{t}))$:
$$W_0(\theta)=W(\Phi_0),\,W_1(\theta)=\{w\in W(\Phi(\mathfrak{t}))\,|\,w(X_0(\theta))=X_0(\theta)\}.$$
\begin{theorem}[\cite{FH}, Proposition 1]\label{thm:little Weyl} In the notation above 
$$W(\mathfrak{a})=W_1(\theta)/W_0(\theta).$$
\end{theorem}
\noindent We will follow the usual terminology and call $W(\mathfrak{a})$ the little Weyl group. Thus, Theorem \ref{thm:little Weyl} describes the relation between the little Weyl group and the Weyl group of $\mathfrak{g}$ (which is $W(\Phi)$).

Now let $k=\mathbb{R}$ and assume $\mathfrak{g}$ is a non-compact semisimple {\it real} Lie algebra. Let $\mathfrak{g}^c$ be the complexification of $\mathfrak{g}$. Consider a Cartan decomposition $\mathfrak{g}=\mathfrak{k}\oplus\mathfrak{p}$ and the corresponding Cartan involution $\theta$. Choose a maximal split abelian subspace $\mathfrak{a}\subset\mathfrak{p}$. Let $W(\mathfrak{a})$ be the little Weyl group.  Extend $\theta$ onto $\mathfrak{g}^c$ by linearity. Then
\begin{itemize}
\item $\mathfrak{g}^c=\mathfrak{k}^c\oplus\mathfrak{p}^c, \mathfrak{k}^c=(\mathfrak{g}^c)^{\theta},\,\mathfrak{p}^c=(\mathfrak{g}^c)_{-}$,
\item one can choose a Cartan subalgebra of $\mathfrak{g}^c$ in the form
$$\mathfrak{t}^c=\mathfrak{t}_{+}^c\oplus\mathfrak{a}^c$$
\item $\mathfrak{t}^c(\mathbb{R})=(i\mathfrak{t}_{+})\oplus\mathfrak{a}$
\end{itemize}
On the other hand, considering $\mathfrak{g}^c$ and $\theta\in\operatorname{Aut}(\mathfrak{g}^c)$ one can obtain $\Phi(\mathfrak{g}^c)$, $\Phi(\mathfrak{a}^c)$ and the Weyl group $W(\mathfrak{a}^c)$.
\begin{lemma}\label{lemma:littleWcompl} $W(\mathfrak{a^c})\cong W(\mathfrak{a})$.
\end{lemma}
\begin{proof} This follows from the explicit construction of the real root system $\Phi(\mathfrak{a})$ in \cite{OV}, p. 155. The set of  roots  $\Phi(\mathfrak{a})$ are constructed in the same way as the set of roots  $\Phi(\mathfrak{a}^c)$ in \cite{FH}. Both sets are obtained as projections of $\mathfrak{t}^c(\mathbb{R})^*$ onto $\mathfrak{a}^*$.  

\end{proof}

\subsection{Weighted Dynkin diagrams of nilpotent orbits}

Given an $\mathfrak{sl}(2,\mathbb{C})$-triple we may assume that $h$ is in the closed positive Weyl chamber of $\mathfrak{g}^{c} .$ 
\begin{definition} The weighted Dynkin diagram of the triple $(h,e,f)$ {\rm is the weighted Dynkin diagram $\psi_h$.}
\end{definition}

\begin{theorem}[\cite{ok}, Proposition 7.8]
The complex adjoint orbit through $e$ meets $\mathfrak{g}$ if and only if the weighted Dynkin diagram of $(h,e,f)$ matches the Satake diagram of $\mathfrak{g} .$
\label{thok}
\end{theorem}

\subsection{De Graaf-Marrani  database and algorithms}
In \cite{dg} the database of all real forms of embeddings of maximal reductive subalgebras of the complex simple Lie algebras of rank up to $8$ was created. We refer to this article for the details of the algorithms and implementations. Here we only make several remarks showing how to use the database.  We say that a pair of real Lie algebras $(\mathfrak{g},\mathfrak{h})$ is a real form of an embedding of complex reductive Lie algbras $\mathfrak{g}^c,\mathfrak{h}^c$, if $\mathfrak{h}^c\subset\mathfrak{g}^c$ and $\mathfrak{h}\subset\mathfrak{g}$. Note that it is possible to get a maximal real semisimple subalgebra $\mathfrak{h}$ in $\mathfrak{g}$ without $\mathfrak{h}^c$ being maximal in $\mathfrak{g}^c$, so only the real forms of embeddings of maximal complex semisimple subalgebras are considered. Algorithms which create the database distinguish between regular and non-regular subalgebras as follows.
\begin{itemize}
\item A subalgebra $\mathfrak{h}^c\subset\mathfrak{g}^c$ is called {\it regular}, if there exists a Cartan subalgebra $\mathfrak{t}^c\subset\mathfrak{g}^c$ such that $[\mathfrak{t}^c,\mathfrak{h}^c]\subset\mathfrak{h}^c$.
\item A subalgebra $\mathfrak{h}^c\subset\mathfrak{g}^c$ is $R$-subalgebra, if it is contained in a proper regular subalgebra in $\mathfrak{g}^c$,
\item A subalgebra $\mathfrak{h}^c\subset\mathfrak{g}^c$ is called $S$-subalgebra, if it it is not an $R$-subalgebra.
\end{itemize} 
Algorithms of creating the database are based on the following description of the real forms of complex embeddings $\mathfrak{h}^c\subset\mathfrak{g}^c$. A real form $\mathfrak{g}$ is given by three maps $\tau,\sigma,\theta:\mathfrak{g}^c\rightarrow\mathfrak{g}^c$ where
\begin{itemize}
\item $\tau,\sigma$ are anti-involutions, $\theta$ is an involution,
\item $\theta=\tau\sigma=\sigma\tau,$
\item $(\mathfrak{g}^c)^{\tau}$ is compact, $(\mathfrak{g}^c)^{\sigma}=\mathfrak{g}$,
\item $\theta$ is a Cartan involution,
\item $\theta$ leaves $\mathfrak{u}=(\mathfrak{g}^c)^{\tau}$ invariant with the eigenspaces $\mathfrak{u}_1\oplus\mathfrak{u}_{-1}$ and 
$$\mathfrak{k}=\mathfrak{u}_1,\mathfrak{p}=\mathfrak{u}_{-1}.$$
\end{itemize} 
The algorithm for enumerating  regular subalgebras is given in \cite{dg} (Section 4.1). The real $S$-subalgebras are classified in a different manner. Let $G^c$ denote the adjoint group of $\mathfrak{g}^c$.  Denote by $G$ the group of all $g\in G^c$ preserving $\mathfrak{g}$, and by $G_0$ the identity component of $G$. There is no direct way to list all  $S$-subalgebras up to conjugacy in $G_0$. Therefore, the algorithms are based on the following approach. Let $\varepsilon: \mathfrak{h}^c\hookrightarrow \mathfrak{g}^c$ be an embedding of semisimple complex Lie algebras and let $\mathfrak{g}$ be a real form of $\mathfrak{g}^c$. Find (up to isomorphism) the real forms of $\mathfrak{g}$ of $\mathfrak{g}^c$ such that $\varepsilon(\mathfrak{h})\subset\mathfrak{g}$. Since any two real forms $\mathfrak{g}$ and $\mathfrak{g}'$ of $\mathfrak{g}^c$ are isomorphic if and only if they are conjugate by an automorphism $\varphi\in\operatorname{Aut}(\mathfrak{g}^c)$ \cite{OV}, we may replace the given embedding by $\varphi\circ\varepsilon$.  
The algorithm is based on the following.
\begin{proposition}[\cite{dg}, Proposition 4.1]\label{prop:real-f} Let $\mathfrak{g}\subset\mathfrak{g}^c$ be a real form of $\mathfrak{g}^c$ such that $\varepsilon(\mathfrak{h})\subset\mathfrak{g}$. Then there are a compact real form $\mathfrak{u}$ of $\mathfrak{g}^c$, with conjugation $\tau:\mathfrak{g}^c\rightarrow\mathfrak{g}^c$ and an involution $\theta$ of $\mathfrak{g}^c$ such that
\begin{enumerate}
\item $\varepsilon(\mathfrak{u}_h)\subset\mathfrak{u}$,
\item $\varepsilon\theta_h=\theta\varepsilon_h$,
\item $\theta\tau=\tau\theta$,
\item there is a Cartan decomposition $\mathfrak{g}=\mathfrak{k}\oplus\mathfrak{p}$ such that the restriction of $\theta$ onto $\mathfrak{g}$is the corresponding Cartan involution and $\mathfrak{u}=\mathfrak{k}\oplus i\mathfrak{p}$.
\end{enumerate}
Conversely, if $\mathfrak{u}\subset\mathfrak{g}$ is a compact real form with corresponding conjugation $\tau$ and involution $\theta$ such that $(1),(2),(3)$ hold, then $\theta$ leaves $\mathfrak{u}$ invariant and setting $\mathfrak{k}=\mathfrak{u}_1, \mathfrak{p}=i\mathfrak{u}_{-1}$ we get that $\mathfrak{g}=\mathfrak{k}\oplus\mathfrak{p}$ is a real form $\mathfrak{g}^c$ with $\varepsilon(\mathfrak{h})\subset\mathfrak{g}$.
\end{proposition}
Now the algorithms and implemetations in \cite{dg} are as follows. Fix a compact real form $\mathfrak{u}$ of $\mathfrak{g}^c$ and replace $\varepsilon$ by $\varphi\varepsilon$ for a $\varphi\in\operatorname{Aut}(\mathfrak{g}^c)$ to get $\varepsilon(\mathfrak{u})\subset\mathfrak{u}$. Construct the space
$$\mathcal{A}=\{A\in\operatorname{End}(\mathfrak{g}^c)\,|\,A(\operatorname{ad}(\varepsilon\theta(y)))=(\operatorname{ad}\varepsilon(y))A\,\,\forall y\in\mathfrak{g}^c\}.$$
If $\theta$ is an involution of $\mathfrak{g}^c$, then $\varepsilon\theta_h=\theta\varepsilon$ if and only of $\theta\in\mathcal{A}$ and 
$$\theta(\operatorname{ad}x)\theta=\operatorname{ad}\theta(x),\,\forall x\in\mathfrak{g}^c.$$
Hence $\mathcal{A}$ contains $\theta$. The conditions on $\theta$ given by Proposition \ref{prop:real-f} are translated to polynomial equations on the coefficients of $\theta$ with respect to a basis of $\mathcal{A}$. These polynomial equations are solved by the technique of the Gr\"obner bases. The details are given in \cite{dg1} and \cite{dg}.

\section{Homogeneous spaces determined by maximal subgroups and proper actions of non-virtually abelian discrete subgroups}
\subsection{General procedures}
\begin{theorem}\label{thm:procedure-split} Let $G/H$ be a homogeneous space of simple real Lie group $G$. Assume that $\mathfrak{g}$ and $\mathfrak{h}$ are split real forms of simple complex Lie algebras $\mathfrak{g}^c$ and $\mathfrak{h}^c.$ The following procedure decides whether $G/H$ admits a proper action of a Lie group $L$ locally isomorphic to $SL(2,\mathbb{R})$.



\begin{enumerate}
	\item List all the non-trivial weighted Dynkin diagrams of $\mathfrak{g}^{c}$ corresponding to $\mathfrak{sl}(2,\mathbb{C})$-triples . Obtain the set \textbf{WDD} of such diagrams;
	\item compute the Weyl group $W(\mathfrak{g}^c)$; 
	\item determine  $H\in \mathfrak{a}^{c}$  given by the weights of the weighted Dynkin diagram in \textbf{WDD};
	\item check if there exists a weighted Dynkin diagram in \textbf{WDD} for which $Wh\cap \mathfrak{a}_h^{c} =\{0\} .$
\end{enumerate}
\end{theorem}
\begin{proof} The proof follows from Theorem \ref{critkob} and Lemma \ref{lemma:littleWcompl}, together with an observation that since $\mathfrak{g}$ is split, any $\mathfrak{sl}(2,\mathbb{C})$-orbit in $\mathfrak{g}^c$  meets $\mathfrak{g}$.  One takes into consideration that $W(\mathfrak{a}^c)=W(\mathfrak{g}^c)$. Indeed, one needs to check that $W(\mathfrak{a})(\mathfrak{a}_l)\cap\mathfrak{a}_h=\{0\}$, which is equivalent to 
$W(\mathfrak{a}^c)(\mathfrak{a}_l^c)\cap\mathfrak{a}_h^c=\{0\}$. In our case $\mathfrak{l}^c$ is generated by $\mathfrak{sl}(2,\mathbb{C})$-triple $(h,e,f)$, and therefore $\mathfrak{a}_l^c=\langle h\rangle$, where $h$ is determined by the weighted Dynkin diagram. This follows from Theorem \ref{thok}, since in the split case the set of the weighted Dynkin diagrams   \textbf{WDD} coincides with the set of the weighted Dynkin diagrams matching the Satake diagram. Therefore, any $\mathfrak{sl}(2,\mathbb{C})$-orbit meets $\mathfrak{g}$, and $h\in\mathfrak{a}\subset\mathfrak{a}^c$. 
\end{proof}
\begin{remark}{\rm It is conceivable, that one could  generalize this procedure to the non-split cases by applying more general algorithms developed in \cite{FH}, however, it does not seem straightforward.}
\end{remark}

\subsection{Proper actions of non-virtually abelian discrete subgroups on homogeneous spaces of small rank}
\begin{theorem}\label{thm:maximal8} Let $G/H$ be a homogeneous space of an absolutely  simple real Lie group of $\operatorname{rank}\,G\leq 8$ over a maximal semisimple subgroup.  Then $G/H$  admits proper actions of a non-virtually abelian discrete subgroup of $G$ if and only of $\operatorname{rank}_{a-hyp} \mathfrak{g} > \operatorname{rank}_{\mathbb{R}} \mathfrak{h}$ or $G/H$ is determined by one of the following pairs of $(\mathfrak{g},\mathfrak{h}$) (with the inclusion $\mathfrak{h}\hookrightarrow\mathfrak{g}$ classified up to conjugation):

$$(\mathfrak{sl}(6,\mathbb{R}), \mathfrak{sl}(2,\mathbb{R})\oplus\mathfrak{sl}(3,\mathbb{R})),$$
$$(\mathfrak{sl}(6,\mathbb{R}), \mathfrak{sl}(4,\mathbb{R})),$$
$$(\mathfrak{e}_{6(6)}, \mathfrak{sl}(3,\mathbb{R})\oplus \mathfrak{g}_{2(2)}),$$ 
$$(\mathfrak{sl}(8,\mathbb{R}), \mathfrak{sl}(2,\mathbb{R})\oplus \mathfrak{sl}(4,\mathbb{R})),$$ 
$$(\mathfrak{sl}(9,\mathbb{R}), \mathfrak{sl}(3,\mathbb{R})\oplus \mathfrak{sl}(3,\mathbb{R})).$$
\end{theorem}
\begin{proof} The proof follows from the classification given in \cite{dg}, Theorem \ref{th8}, Theorem \ref{thm:procedure-split} and the computer implementation of two procedures. The first one checks the conditions on the properness given by Theorem \ref{th8} for any pair $(\mathfrak{g},\mathfrak{h})$  in the database given in \cite{dg}. 
We have four types of pairs: $L_0$ ($\operatorname{rank}_{\mathbb{R}}\mathfrak{g}=\operatorname{rank}_{\mathbb{R}}\mathfrak{h}$), $L_1$ ($\operatorname{rank}_{a-hyp}\mathfrak{g}=\operatorname{rank}_{a-hyp}\mathfrak{h}$), 
$L_2$ ($\operatorname{rank}_{a-hyp}\mathfrak{g}>\operatorname{rank}_{\mathbb{R}}\mathfrak{h})$)
 and $L_3,$ the remaining cases. At the beginning, we remove all pairs of type $L_0$, because in this case only finite subgroup can act properly on $G/H$ (the Calabi-Markus
phenomenon). Next we remove all pairs of type $L_1$, since they do not admit proper actions of non-virtually abelian discrete subgroups by Theorem \ref{th8}. We store all  pairs $(\mathfrak{g},\mathfrak{h})$ of type $L_2$ , because  they correspond to homogeneous spaces admitting proper actions of non-virtually abelian discrete subgroups. Next, we consider pairs of type $L_3$ and observe that all such pairs are split. Therefore, we are able to apply the procedure given by Theorem \ref{thm:procedure-split}. We keep in mind that we are dealing with the complexifications  of the pairs $(\mathfrak{g},\mathfrak{h})$, and, therefore, may assume that $\mathfrak{a}_h^c\subset\mathfrak{a}^c$ and $\mathfrak{a}_l^c\subset\mathfrak{a}^c$ for representatives of the database \cite{dg} (see  remarks below).
\end{proof}
\subsection{Remarks on  Theorem \ref{thm:maximal8}}
Note that the condition $\operatorname{rank}_{a-hyp}\mathfrak{g}>\operatorname{rank}_{\mathbb{R}}\mathfrak{h}$ {\it does not} tell us, what is $\Gamma$ acting properly on $G/H$. We only know that some $\Gamma$ does exist. On the other hand, the class $L_3$ consists of homogeneous spaces which admit a proper $L$-action, where $L$ is locally isomorphic to $SL(2,\mathbb{R})$. It is worth noting that it is a difficult open  problem to decide if there exists $\Gamma$ which is not a discrete subgroup of $SL(2,\mathbb{R})$ and which acts properly on $G/H$. Actually there are two examples of homogeneous spaces $G/H$ which admit a proper action of a discrete subgroup $\Gamma\subset G$, but do not admit a proper action of $L$ locally isomorphic $SL(2,\mathbb{R})$ \cite{b}. 
  
 It should be stressed, that the condition $\operatorname{rank}_{a-hyp}\mathfrak{g}>\operatorname{rank}_{\mathbb{R}}\mathfrak{h}$ can be checked in a straightforward manner using Table 1. Since the de Graaf-Marrani database contains {\it 134 tables} of various  pairs $(\mathfrak{g},\mathfrak{h})$ it does not seem reasonable to reproduce the whole class $L_2$. One can use Table 1  and properties of the a-hyperbolic rank, or use our  plugins with additional functions calculating the a-hyperbolic ranks and real ranks (see Section \ref{sect:implement}).

 In general, working with the database \cite{dg} we should keep in mind the following difficulty. The tables in \cite{dg} are created by presenting members of conjugacy classes of subalgebras $\mathfrak{h}\hookrightarrow\mathfrak{g}$. Therefore, it may happen, that the representative in the table does not satisfy the requirement $\mathfrak{a}_h\subset\mathfrak{a}$ (and this is indeed the case, with some exceptions).

\section{Implementation}\label{sect:implement}

The procedure of obtaining lists of pairs of types $L_{0}-L_{3}$ from the database in \cite{dg} is straightforward. Therefore we describe the implementation of the algorithm which checks if a given split pair $(\mathfrak{g},\mathfrak{h})$ corresponds to a homogeneous space admitting a proper action of a subgroup $L$ locally isomorphic to $SL(2,\mathbb{R}).$ Obviously if a homogeneous space $G/H$ admits a proper action of $L$ it also admits a proper action of a non-virtually abelian discrete subgroup (one can take a co-compact lattice of $L,$ for instance). 

\noindent
We have implemented the above procedures in the computer algebra system GAP \cite{gap} and the  following two plugins: SLA \cite{sla}, CoReLG \cite{corelg}. We have also updated a special plugin CKForms \cite{ckforms}. In version 2.0 there are additional functions that implement a-hyperbolic rank (\texttt{AHypRank}). The final database of pairs is based on the data from the CoReLG \cite{corelg} plugin extended with calculations on ranks.
\subsection{Algorithm for checking the Kobayashi criterion}
\begin{algorithm}[H]
  \caption{\tt CheckProperSL2RAction($\mathfrak{g},\mathfrak{h}$)}
  \footnotesize
  \label{alg1}
  \tcc{
  $\mathfrak{g}$ - non-compact real simple Lie algebra, $\mathfrak{h}$ - maximal proper subalgebra. 
   Return true, when corresponding $G/H$ admits a proper action of a subgroup $L\subset G$ locally isomorphic to $\mathfrak{sl}(2,\mathbb{R})$ and return false - otherwise.}
  \Begin{
let $\mathfrak{j}_{\mathfrak{h}}$ be Cartan subalgebra of $\mathfrak{h}^c$\;
let $W$ be Weyl group $\mathfrak{g}^c$ represented by $r\times r$-matrix ($r=\operatorname{rank} \mathfrak{g}^c$)\;
set $C_\mathfrak{h}$ - the basis vectors of $\mathfrak{j}_{\mathfrak{h}}$ in Chevalley basis of $\mathfrak{g}^c$\;
set $Orb$ - the set of nilpotent orbits of $\mathfrak{g}^c$\;
 \ForAll{$o \in Orb$}{
set $h_o$ as vector $H$ from $(h,e,f)$-$SL2$-triple for $o$\;
\If{ each vector from $ W h_0$ is linearly independent with $C_\mathfrak{h}$  }
{\Return{true}}
 }
   \Return{false}\;
  
  }
\end{algorithm} 

The keyword "return"  in the above terminates the algorithm.



\vskip20pt
Faculty of Mathematics and Computer Science
\vskip6pt
\noindent University of Warmia and Mazury
\vskip6pt
\noindent S\l\/oneczna 54, 10-710 Olsztyn, Poland
\vskip6pt
\noindent e-mail adresses:
\vskip6pt
mabo@matman.uwm.edu.pl (MB)
\vskip6pt
piojas@matman.uwm.edu.pl (PJ)
\vskip6pt
tralle@matman.uwm.edu.pl (AT)

\end{document}